\documentclass[11pt,a4paper]{amsart}

\title[Pseudomeromorphic currents on subvarieties]{Pseudomeromorphic currents on subvarieties}
\author{Mats Andersson}
\thanks{The author was partially supported by a grant from the Swedish Research Council.} 
\address{Department of Mathematical Sciences, Division of Mathematics, University of Gothenburg and 
Chalmers University of Technology, SE-412 96 G\"{o}teborg, Sweden}
\email{matsa@chalmers.se,  larkang@chalmers.se}

\usepackage{amsmath}
\usepackage{amsthm,amssymb,latexsym}
\usepackage{url,ae}
\usepackage{t1enc}
\usepackage{mathrsfs}
\usepackage{graphicx}
\usepackage{eufrak} 
\usepackage{geometry}
\newtheorem{thm}{Theorem}[section]
\newtheorem{lma}[thm]{Lemma}

\newtheorem{prop}[thm]{Proposition}

\theoremstyle{definition}

\theoremstyle{remark}

\newtheorem{preremark}{Remark}
\newtheorem{preex}{Example}

\newcommand{\Ker}{{\text{Ker}\,}}
\newcommand{\C}{\mathbb{C}}
\newcommand{\debar}{\bar{\partial}}
\newcommand{\dbar}{\bar{\partial}}

\newcommand{\R}{\mathbb{R}}

\newcommand{\E}{\mathscr{E}}

\newcommand{\W}{\mathcal{W}}
\newcommand{\V}{\mathcal{V}}
\newcommand{\PM}{\mathcal{PM}}

\newcommand{\D}{\mathscr{D}}

\newcommand{\pmm}{pseudomeromorphic }
\newcommand{\nbh}{neighborhood }
\newcommand{\1}{{\bf 1}}
\newcommand{\w}{{\wedge}}

\newcommand{\re}{{\text{Re}\,}}

\newcommand{\U}{{\mathcal U}}
\newcommand{\Cu}{{\mathcal C}}

\def\newop#1{\expandafter\def\csname #1\endcsname{\mathop{\rm #1}\nolimits}}
\newop{span}

\date{2014-01-01}

\numberwithin{equation}{section}

\begin{document}
\nocite{*}
\bibliographystyle{plain}

\begin{abstract}
Let $i\colon X \to Y$ be pure-dimensional reduced subvariety of a smooth manifold $Y$.
We prove that the direct image of pseudomeromorphic currents on $X$ are 
pseudomeromorphic on $Y$.  We also prove  a partial converse:  if $i_*\tau$ is pseudomeromorphic
and has the standard extension property, then $\tau$ is pseudomermorphic on $X$.
\end{abstract}

\maketitle
\thispagestyle{empty}

\section{Introduction}

Let $X$ be a pure-dimensional analytic space.  In \cite{AW2} was introduced the sheaf 
$\PM^X$ of pseudomeromorphic currents, and the definition was somewhat further widened
in \cite{AS}.  The principal examples are semi-meromorphic forms and $\dbar$ of such forms,
as well as direct images under modifications, natural projections, and open inclusions, of such
currents. 

The interest of this sheaf relies on two facts. To begin with, many currents that occur in multivariable residue theory are
pseudomeromorphic; for instance Coleff-Herrera products, \cite{CoHe}, the more general
Coleff-Herrera currents, \cite{Bj},  
 Bochner-Martinelli type currents, introduced in \cite{PTY}, and for instance the currents introduced in \cite{A1}
and \cite{AW1}.   Moreover,   pseudomeromorphic currents have some  "geometric" properties that are similar
to  basic properties of positive closed $(*,*)$-currents.   For instance, for each analytic subvariety $V\subset X$ 
and pseudomeromorphic  current $\mu$ on $X$, the natural restriction of $\mu$ to $X\setminus V$ has a canonical 
pesudomeromorphic extension $\1_{X\setminus V}\mu$ to $X$, and 
\begin{equation}\label{anita}
\1_V\mu:=\mu- \1_{X\setminus V}\mu
\end{equation}
 is pseudomeromorphic and has support on $V$.  If $V'$ is another subvariety, then 
\begin{equation}\label{rakneregel}
\mathbf{1}_V \mathbf{1}_{V'} \mu = \mathbf{1}_{V\cap V'} \mu.
\end{equation} 
Moreover, we have the  {\it dimension principle}, 
that states that if $\tau$ is a \pmm $(*,p)$-current with support on
an analytic set with codimension larger than $p$, then $\tau$ must vanish.
These basic properties very useful or even indispensable  tools in, for instance, 
\cite{AW2, ASS, AS, AWsemester, Lark1, Lark2, Lund, Sz}.

%%Moreover, if $\tau$ is $\dbar$-closed and has support on a pure-dimensional
%%set of codimension $p$, then $\tau$ is a Coleff-Herrera current. In particular this implies
%%that its annihilator is a pure $p$-codimensional coherent analytic ideal sheaf. 

If $\mu$ is pseudomeromorphic and has support on a pure-dimensional subvariety $V\subset X$ we say that
$\mu$ has the  {\it standard extension property}, SEP, with respect to $V$,  if 
$\1_A\mu=0$ for each subvariety $A\subset V$ of positive codimension. 
%%
%%If $V$ has codimension $p$ and $\mu$ has support on $V$ and bidegree $(*,p)$, then it has the
%%SEP in view of the dimension principle.
 We let $\W^X_V$ denote the sheaf of \pmm currents on $X$ with support and the SEP on $V$.

\smallskip

Assume that 
$i\colon X\to Y$ is an embedding of a reduced pure-dimensional space $X$ into a smooth manifold $Y$.
Recall that the sheaf of smooth forms on $X$ is defined as the quotient sheaf $\E^X:=\E^Y/\Ker i^*$.
The image of $\xi$ in $\E^X$ is denoted by $i^*\xi$.
By definition $\tau$ is a current on $X$, $\tau$ in $\Cu^X$,  if it is in the dual of $\E^X$. This means that there is 
a current $\mu$ on $Y$ with support on $X$
such that $\xi\wedge \mu=0$ for all test forms $\xi$ such that $i^*\xi=0$, so that $\tau.i^*\xi:=\mu.\xi$.
It is therefore natural to write $\mu=i_*\tau$.    There is an induced $\dbar$-operator on forms
and currents on $X$.  
Here is our main result in this note.

\begin{thm}\label{L1}
 Assume that 
$i\colon X\to Y$ is an embedding of a reduced pure-dimensional space $X$ into a smooth manifold $Y$.

\smallskip
\noindent (i) \  If $\tau$ is in $\PM^X$,   then $i_*\tau$ is in $\PM^Y$, 
and if $\tau$ is in $\W^X$ then $i_*\tau$ is in $\W^Y_X$.

\smallskip
\noindent (ii) \  If $\tau$ is in $\Cu^X$ and $i_*\tau$ is in $\PM^Y$, and in addition,
\begin{equation}\label{pia2}
\1_{X_{sing}}i_*\tau=0,
\end{equation}
then $\tau$ is in $\PM^X$. If $i_*\tau$ is in $\W^Y_X$, then
$\tau$ is in $\W^X$.
\end{thm}

That is,  we have the natural mappings
$$
i_* \colon \PM^X\to \PM^Y,  \quad i_*\colon \W^X\to \W^Y_X.
$$
Notice that the condition \eqref{pia2} in (ii)  is  automatically fulfilled if $i_*\tau$ is in $\W^Y_X$.

The proof of Theorem~\ref{L1} relies  very much on the existence of a so-called 
strong desingularization, see below. However we also need
the following  result which is interesting in itself. 
%%We say that the natural projection $X\times Z\to X$ 
%%is a {\it simple projection}.  

\begin{prop}\label{kraka}
If $p\colon X'\to X$ is a  modification and $X'$ is smooth,
then %%  and $\U\subset X$ is an open subset, then
$$
p_*\colon\PM(X')\to\PM(X)
$$
is surjective.
\end{prop}

%%In the same way, if

\section{Pseudomeromorphic currents}

Recall that in one complex variable $t$ one can define  the {\it principal value current}  $1/t^m$, $m\ge 1$, 
as the value at $\lambda=0$ of the analytic continuation of $|t|^{2\lambda}/t^m$,
a~priori defined when $\re\lambda\gg 0$. The {\it residue current}
$\dbar(1/t^m)$ is the value at $\lambda=0$ of $\dbar|t|^{2\lambda}/t^m$; clearly it has  support
at $t=0$.

Assume now that $t_j$ are holomorphic coordinates in an open set $U\subset\C^N$. 
Since we can take tensor products of one-variable currents, we can form
the current %%%  (cf., also Remark~\ref{chprod} below)
\begin{equation}\label{element}
\tau=\debar \frac{1}{t_1^{a_1}}\wedge \cdots \wedge \debar \frac{1}{t_r^{a_r}}\wedge 
\frac{\alpha(t)}{t_{r+1}^{a_{r+1}}\cdots t_N^{a_N}},
\end{equation}
where $a_1,\ldots,a_r$ are positive integers, $a_{r+1},\ldots,a_N$ are nonnegative
integers, and  $\gamma$ is a smooth form with compact support in $U$. Such a $\tau$
is called an {\em elementary (pseudomeromorphic) current} in $U$.  It is  commuting
in the principal value factors and anti-commuting in the residue factors.

\smallskip

Fix a point $x\in X$. We say that a germ $\mu$  of a current at $x$ is {\it pseudomeromorphic} at $x$,
$\mu\in \PM_x$,  if it is a finite sum
of currents of the form
$\pi_*\tau=\pi_*^1 \cdots \pi_*^m \tau$,  where $\U$ is a \nbh of $x$, 
\begin{equation}\label{1struts}
\U^m \stackrel{\pi^m}{\longrightarrow} \cdots \stackrel{\pi^2}{\longrightarrow} \U^1 
\stackrel{\pi^1}{\longrightarrow} \U^0=\U,
\end{equation}
each $\pi^j\colon \U^j \to \U^{j-1}$ is either a modification,  a simple projection
$\U^{j-1}\times Z \to \U^{j-1}$, or an open inclusion (i.e., $\U^{j}$ is an open subset
of $\U^{j-1}$), and $\tau$ is elementary on $\U^m$.
%%\end{definition}

By definition the union $\PM=\cup_x\PM_x$ is an open subset of the sheaf $\Cu=\Cu^X$
and hence it is a subsheaf, the  sheaf of {\it pseudomeromorphic}  currents,  of $\Cu$.
A section $\mu$ of $\PM$ over an open set $\V\subset X$,  $\mu\in\PM(\V)$, is then a locally finite
sum 
\begin{equation}\label{batting}
\mu=\sum (\pi_\ell)_*\tau_\ell,
\end{equation}
where each $\pi_\ell$ is a composition of mappings as in \eqref{1struts}
(with $\U\subset\V$) and $\tau_\ell$ is elementary.  
%%%
The definition here is from \cite{AS} and it is in turn a slight elaboration of the
definition introduced in \cite{AW2}.

 If $\xi$ is a smooth form, then 
$\xi\w \pi_*\tau=
\pi_*\big(\pi^*\xi\w\tau\big).
$
Thus $\PM$ is closed under exterior multiplication by smooth forms.
%%%
Notice that if $\tau$ is an elementary   current, then $\dbar\tau$ is a finite
sum of elementary currents. Since moreover $\dbar$ commutes with push-forwards
it follows that $\PM$ is closed under $\dbar$. 
%%In  the same way $\PM$  is
%%closed under $\partial $.

Assume that  $\mu$ is \pmm  and $V$ is a subvariety. Let $h$ be a tuple of holomorphic functions such that
the common zero set is precisely $V$.  The  function
$\lambda\mapsto|h|^{2\lambda}\mu$ (a priori defined for $\re \lambda \gg 0$) has a current-valued analytic continuation
to $\re \lambda >-\epsilon$.  The value at $\lambda=0$ is precisely the pseudomeromorphic current $\mathbf{1}_{X\setminus V}\mu$ mentioned above, 
and we  write 
\begin{equation}\label{restrikdef}
\mathbf{1}_{X\setminus V}\mu = |h|^{2\lambda} \mu |_{\lambda =0}.
\end{equation} 
One can also obtain $\mathbf{1}_{X\setminus V}\mu$ as a principal value: If $\chi$ is a smooth approximand
 of the characteristic function of $[1,\infty)$ on $\R$, then 
\begin{equation}\label{restrikdef2}
\mathbf{1}_{X\setminus V}\mu = \lim_{\delta\to 0^+} \chi(|h|/\delta) \mu.
\end{equation}
Notice that $\1_V \mu=\mu$   if $\mu$ has support on $V$,
 cf., \eqref{anita}.
%%It follows that the current 
%%\begin{equation*}
%%\mathbf{1}_V \mu :=\mu-\mathbf{1}_{X\setminus V}\mu
%%\end{equation*}
%%is in $\PM^X$ and has support on $V$, and in particular that %%Moreover, cf.,  \eqref{restrikdef},
%%$\1_V\mu=\mu$ if $\mu$ has support on $V$.
%%
The existence of \eqref{restrikdef} and the independence of $h$ follow from the corresponding statements 
for elementary currents, 
noting that if $\mu =\pi_* \tau$, then $|h|^{2\lambda} \mu =\pi_* (|\pi^* h|^{2\lambda}\tau)$
for $\re \lambda \gg 0$.  %%%, see \cite{AW2, AS}.
%%\end{equation}
In the same way one can reduce the verification of \eqref{restrikdef2} to the case
with elementary currents.
Notice that if $p$ is a modification or simple projection, then, cf.,  \eqref{restrikdef},
\begin{equation}\label{pia3}
\xi\w p_*\tau=p_*(p^*\xi\w\tau), \quad  \1_V p_*\tau=p_* \big(\1_{p^{-1}V} \tau\big).
\end{equation}

%%\smallskip
If $\tau$ is \pmm and has support on  $V$, and  $h$ is a holomorphic function that
vanishes on $V$, then $\bar h \tau=0$ and $d\bar h\w \tau=0$, see \cite{AW2, AS}.
This intuitively means 
that the current $\tau$ only involves holomorphic derivatives of test forms.

 %%\section{Proof of Proposition \ref{kraka}}

\section{Proofs} %% of Proposition \ref{kraka}}

\begin{lma}\label{1doda}
Assume that  $\tau$ is an elementary current of the form \eqref{element}.  Let
$t^b=t_1^{b_1}\cdots t_r^{b_r}$ be a monomial and $\gamma$ a strictly positive smooth function.  
 Then
$$
\frac{|t^b|^{2\lambda}\gamma^\lambda}{t^b} \tau, \quad
\frac{\dbar\big(|t^b|^{2\lambda}\gamma^\lambda\big)}{t^b} \w \tau, 
$$
both have analytic continuation to $\re\lambda>-\epsilon$, and  the values at $\lambda=0$ are
elementary  \pmm currents that are  independent of $\gamma$, 
\end{lma}

\begin{proof}
 First assume that  $\gamma=1$.  Then the lemma is 
basically a one-variable statement, and  follows from the observation that
$$
\lambda\mapsto \frac{|t^b|^{2\lambda}}{t^b} \frac{\alpha}{t^m}, \quad
\lambda\mapsto \frac{\dbar |t^b|^{2\lambda}}{t^b} \w\frac{\alpha}{t^m},
$$
admit   the desired analytic continuations, and that the values at $\lambda=0$ are the currents $\alpha/t^{m+b}$
and $\dbar(1/t^{m+b})\w\alpha$, respectively,
together with the trivial fact that
$$
\lambda\mapsto \frac{|t^b|^{2\lambda}}{t^b} \alpha \w\dbar\frac{1}{t^m}=0,
\quad 
\lambda\mapsto \frac{\dbar |t^b|^{2\lambda}}{t^b} \w\alpha \w\dbar\frac{1}{t^m}=0,
$$
when $\re\lambda \gg 0$.

When $\gamma$ is just  strictly positive  we introduce the complex parameter
$\mu$ and notice that 
$$
\lambda,\mu\mapsto 
\frac{|t^b|^{2\lambda}\gamma^{\mu}}{t^b} \tau, \quad
\lambda,\mu\mapsto \frac{\dbar|t^b|^{2\lambda}\gamma^{\mu}}{t^b} \w \tau, 
$$
are analytic for  $(\lambda,\mu)\in \{\re\lambda>-\epsilon\}\times\C$. Thus the value at $\lambda=\mu=0$ can be obtained
by first letting $\mu=0$ and then $\lambda=0$, and so we are back to the case when $\gamma=1$.
\end{proof}

%%We begin with a lemma that probably ????

\begin{lma}\label{1ursnygg}
Assume that $p\colon Y\to X\subset\subset\C^n$ is a modification or a simple projection and
$\tau$ is an elementary \pmm current in $X$ (with respect to the standard
coordinates in $\C^n$). Then  there is a modification
$\tilde p\colon\widetilde Y\to Y$ such that
$$
\tau=p_*\tilde p_*\sum_\ell \tau_\ell,
$$
where the sum is finite and each $\tau_\ell$ is elementary with respect to some
local coordinates in  $\widetilde Y$.
%%If $h$ is holomorphic in $Y$ we may assume as well that $\tilde p^*h$ is
%%a monomial times a nonvanishing factor with respect to the same local coordinate
%%systems.
\end{lma}

\begin{proof}
Let us first assume that $p$ is a modification and that
$\tau$ is elementary with respect to the coordinates $t_j$ in  $X$,
say  of the form \eqref{element}.
Notice that $p^*t_j$ are global holomorphic functions in $Y$.
There is a smooth modification $\tilde p\colon \widetilde Y\to Y$ and
an open cover $\U_\ell$ of $\tilde Y$ such that,  for each $\ell$,  all the functions
$\tilde p^*p^*t_j$ are monomials
(with respect to the same local coordinates $s$) times a nonvanishing holomorphic factor
in $\U_\ell$.  
%%We may assume as well that the same holds
%%for $\tilde p^* h$.
Take a partition of unity $\chi_\ell$ subordinate
to $\U_\ell$.
If
$$
\tau^\lambda:=\tau^{\lambda_1,\ldots,\lambda_N}:=
\frac{\dbar|t_{1}|^{2\lambda_1}}{t_{1}^{a_{1}}}\w\ldots
\w \frac{\dbar|t_{r}|^{2\lambda_r}}{t_r^{a_r}} \w
\alpha\frac{|t_{r+1}|^{2\lambda_{r+1}}}{t_{r+1}^{a_{r+1}}}\cdots\frac{|t_N|^{2\lambda_N}}{t_N^{a_N}},
$$
where $N\le n¢$, then
$$
\tau=\tau^{\lambda_1,\ldots,\lambda_N}|_{\lambda_N=0}\cdots|_{\lambda_1=0}.
$$
Let $\pi=\tilde  p\circ p$. For $\lambda\gg 0$ we have that
$$
\pi^*\tau^\lambda=\sum_\ell\chi_\ell\pi^*\tau^\lambda.
$$
By repeated applications of Lemma \ref{1doda} it follows,  for each $\ell$, that
\begin{equation}\label{1uthus2}
\chi_\ell\pi^*\tau^\lambda|_{\lambda_N=0}\cdots|_{\lambda_1=0}
\end{equation}
exists and is a finite sum $\tilde\tau_\ell$  of  elementary  currents in
$\U_\ell$.
Since $\tau^\lambda=\pi_*\pi^*\tau^\lambda$ when $\re\lambda\gg 0$,  we conclude that 
$$
\tau=\pi_*\sum_\ell\tilde\tau_\ell=p_*\tilde p_*
\sum_\ell\tilde\tau_\ell.
%%=p_*\big(\sum_\ell\hat p_*\tilde\tau_\ell\big)=
%%p_*\tilde\tau,
$$
%%and clearly $\tilde\tau=\hat p_*\sum_\ell\tilde\tau_\ell$ is \pmm in $\tilde X$.

%%\vspace{1cm}

\smallskip
If $p$ is a simple projection $X\times X'\to X$, we can take any test form  $\chi$ in $X'$
with total integral $1$. Then the tensor product $\tau\otimes\chi$ is en elementary current
in $X\times X'$ such that $p_*(\tau\otimes\chi)=\tau$.
%%The statement  about $h$ follows now from the first part of the proof.
\end{proof}

The order that we let $\lambda_j$  be $0$ in the proof is arbitrary.
However, the single terms $\tilde\tau_\ell$ in $\tilde Y$, as well as 
the resulting  current $\tilde p\tau$,
will depend on the order.
%%cf., Example~\ref{1gatstump} below,
%%and also  the resulting  current $\tilde p\tau$.

\begin{proof}[Proof of Proposition \ref{kraka}]
Assume that $\mu=\pi_* \tau$, where $\pi$ is a composed mapping as in
\eqref{1struts} and $\tau$ is elementary in $\U_m$. It is enough to see  that $\mu=p_*\mu'$
for some $\mu'\in \PM(\V)$ where $\V=p^{-1}\U$.
The proposition then follows since a general global section for a locally finite sum och such $\mu$ 
since $p$ is proper.
We claim that \eqref{1struts}
can be extended to a  commutative diagram
\begin{equation}\label{kommdiagram}
\begin{array}[c]{ccccccccccc}
\widetilde\V &= & \V_m & \stackrel{\tilde{\pi}_m}{\longrightarrow} & \cdots &
\stackrel{\tilde{\pi}_2}{\longrightarrow}
& \V_1 & \stackrel{\tilde{\pi}_1}{\longrightarrow} & \V_0 & = & \V \\
& &\downarrow \scriptstyle{p_m} & & & & \downarrow \scriptstyle{p_1} & & \downarrow \scriptstyle{p} & & \\
\widetilde U &= & \U_m & \stackrel{\pi_m}{\longrightarrow} & \cdots & \stackrel{\pi_2}{\longrightarrow} &
\U_1  & \stackrel{\pi^1}{\longrightarrow} & \U_0 & = & \U
\end{array}
\end{equation}
so that each vertical map is a modification and each $\tilde{\pi}_j$
is either a modification,   a simple projection, or an open inclusion,
cf., the proof of Proposition~2.7 in \cite{AS}. 
To see this,  assume that this is done up to level $k$.
It is well-known that if $\pi_{k+1}\colon \U_{k+1}\to \U_{k}$ is a modification, then there are
modifications $\tilde{\pi}_{k+1}\colon \V_{k+1}\to \V_{k}$ and
$p_{k+1}\colon \V_{k+1}\to \U_{k+1}$ such that
\begin{equation*}
\begin{array}[c]{ccc}
\V_{k+1} & \stackrel{\tilde{\pi}_{k+1}}{\longrightarrow} & \V_k \\
\downarrow \scriptstyle{p_{k+1}} & & \downarrow \scriptstyle{p_k} \\
\U_{k+1} & \stackrel{\pi^{k+1}}{\longrightarrow} & \U_k
\end{array}
\end{equation*}
commutes. If instead $\U_{k+1}=\U_k \times Z$ then we simply take
$\V_{k+1}=\V_k \times Z$. Finally, if $i\colon\U_{k+1}\to\U_k$ is an open
inclusion, then we take $\V_{k+1}=p_k^{-1}\U_{k+1}$.

By  Lemma~\ref{1ursnygg} there is a \pmm current $\tilde\tau$ with compact support in
$\V_m$ such that $p_m\tilde\tau=\tau$. If $\tilde\pi$ is the composed
mapping in the upper line, it follows that $\mu'=\tilde\pi_*\tilde\tau$ is
\pmm in $\V$ such that $p_*\mu'=\mu$.

%%\smallskip
%%If  now instead $p$ is a simple projection $Y=X\times Z\to X$, then
%%let $\mu'$ be the tensor
%%product of $\mu$ and a test form  in $Z$ with integral $1$. %%, cf., Lemma~\ref{1tensor}.
\end{proof}

%%The notion of pseudomeromorphic currents on a complex manifold was introduced in \cite{AW2} and
%%somewhat elaborated in \cite{AS}.  
%%the same definition works on  our  singular space $X$. The definition in \cite{AW2}
%%is however somewhat more restrictive as it  does not allow  simple projections.
%%As we will see,
%%all basic properties are preserved also for our class. 
%%\end{remark}

\begin{lma}\label{tensor} 
If $\mu\in \PM(X)$ and $\mu'\in\PM(X')$, then
$\tau\otimes \tau'\in \PM(X\times X')$.
\end{lma}

\begin{proof} It is enough to consider the case $\mu=\pi_*\tau$,
$\mu'=\pi'_*\tau'$, where $\tau,\tau'$ are elementary, and 
$\pi,\pi'$ are compositions of mappings as  in \eqref{1struts}.
However, it is easily verified that then
$$
\pi\otimes\pi'\colon \U_m\times\U_{m'}\to \U\times\U'\subset X\times X'
$$
is again a composition of modifications, simple projections, and open inclusions.
Since $\mu\otimes\mu'=(\pi\otimes\pi')_*\tau\otimes\tau'$  it is
pseudomeromorphic by definition.
\end{proof}

%%\section{Proof of Theorem \ref{L1}}

%%Let $\W(X)$ be the pseudomeromorphic currents on $X$ that have the SEP on $X$.
%%Assume that $j\colon X\to X'$ is an embedding.  Let $\W_X(X')$ be the pseudomeromorphic currents
%%on $X'$ that have support on $X$ and the SEP with respect to $X$.  Notice that we have the mapping
%%\begin{equation}\label{plast}
%%j_*\colon\W(X)\to\W(X',X)
%%\end{equation}

%%\begin{cor} The image of the mapping \eqref{plast} is precisely the space of  $\mu\in\W(X,X')$ such that
%%$\xi\wedge \mu=0$ for all smooth $\xi$ such that $j^*\xi=0$.
%%\end{cor}

As already mentioned the proof of Theorem \ref{L1} relies on the existence of a 
{\it strong desingularization}, 
see, e.g., \cite{BESVO} and the refererences given there. This  means
that there is a smooth modification $p\colon \widetilde Y\to Y$ that is a biholomorphism outside
$X_{sing}$ and such that the strict transform $\widetilde X$ of $X$ is a smooth  submanifold
of $\widetilde Y$ and the restriction $p'$ of $p$ to $\widetilde X$ is a modification
$p'\colon \widetilde X\to X$ of $X$. Thus we have a commutative diagram
\begin{equation}\label{strong}
\begin{array}[c]{ccccccccc}
\widetilde X & \stackrel{\tilde{i}}{\longrightarrow} & \widetilde Y \\
\downarrow \scriptstyle{p'} & &  \downarrow \scriptstyle{p} \\
X  & \stackrel{i}{\longrightarrow} & Y
\end{array}.
\end{equation}

\begin{proof}[Proof of Theorem \ref{L1}]
 First assume that $X$ is a smooth submanifold. The statement (i) is local so
we may assume that $Y=X_z\times \C^r$  and $i(z)=(z,0)$.  It is easily checked that
$i_*\tau$ is equal to the tensor product
\begin{equation}\label{pia}
\mu:=\tau\w[w=0] %:=\tau\otimes\dbar\frac{1}{w}\w dw (2\pi i)^{-r}
\end{equation}
where $[w=0]$ means the point evalutation at $0\in\C^r$. In view of  Lemma~\ref{tensor}
it is then pseudomeromorphic since $[w=0]=\dbar\frac{1}{w}\w dw (2\pi i)^{-r}$ is.
For a test form $\xi=\xi(z,w)$, we can write 
$\xi=\xi'+\xi''$, where $\xi'$ contains no occurrences of $dw_j$ or $d\bar w_j$. 
%%$\xi(z,0)$ denote 
%%the components of $\xi$ that have no differentials with respect to $w$, restricted to $w=0$.
Then
$$
i_*\tau.\xi=\tau.i^*\xi=\tau.i^*\xi'=\tau.\xi'(\cdot,0)=\mu.\xi,
$$
cf., \eqref{pia}, and hence $i_*\tau=\mu$ is pseudomeromorphic in $Y$. 
Now assume that $i\colon X\to Y$ is arbitrary and consider \eqref{strong}. Any $\tau\in\PM(X)$
can be written $p'_*\tilde\tau$ for some $\tilde\tau\in\PM(\widetilde X)$ according to
Proposition \ref{kraka}. By  the first part we now that $\tilde i_*\tilde\tau$ is
pseudomeromorphic in $\widetilde Y$.
Thus %%the general case follows, since
$i_*\tau=i_*p'_*\tilde\tau=p_*\tilde i_*\tilde \tau$ is pseudomeromorphic in $Y$, and
so the first part of (i) is proved.

Assume that  $V\subset X$ has positive codimension.  Since $i^{-1}V=V$ we have, cf., \eqref{pia3},
that $\1_V i_*\tau=i_* \1_V\tau$. Thus $i_*\tau$ is in $\W^Y_X$ if (and only if) $\tau$ is in $\W^X$,
and so the second part of (i) follows.

\smallskip
We now consider (ii).  
Again assume first that $X$ is smooth. Again the statement is local so we may assume that
$Y=X_z\times \C^r_w$.  Let $\pi\colon Y\to X_z$ be the projection  $(z,w)\mapsto z$.
Since $i_*\tau$ is \pmm by assumption also $p_*i_*\tau$ is pseudomeromorphic.  
Now,
$$
p_*i_*\tau.i^*\xi=i_*\tau.p^*i^*\xi=i_*\tau.\xi'(\cdot,0)=\tau.i^*\xi,
$$
for all test forms $\xi$,  and hence $p_*i_*\tau$. We conclude that $\tau$ is in $\PM^X$. 
Thus (ii) holds in case $X\subset Y$ is smooth.

Now assume that $i\colon X\to Y$ is general,  $\mu:=i_*\tau\in\PM(Y)$,
and consider \eqref{strong}. We claim that  
$\mu=p_*\tilde\mu$,
where $\tilde\mu\in\PM(\widetilde Y)$,
$\tilde\mu$ has support on $\widetilde X$, and 
$\1_{p^{-1}X_{sing}}\tilde\mu=0$.  
To begin with $\mu=p_*\hat \mu$ for some $\hat\mu\in\PM(\widetilde Y)$
according to Proposition \ref{kraka}.  
Since 
$$
0=\1_{Y\setminus X}p_*\hat\mu=p_*(\1_{\widetilde Y\setminus  p^{-1}X}\hat\mu),
$$
cf., \eqref{pia3}, we have that $\mu=p_* \mu'$ where $\mu':=\1_{p^{-1}X}\hat\mu$ has support on  $p^{-1}X$.
Notice that this set is in general much larger than the strict transform
$\widetilde X$ of $X$. Now
$$
\mu'=\1_{p^{-1}X_{sing}}\mu'+\1_{p^{-1}(X\setminus X_{sing})}\mu'
$$
and, by assumption \eqref{pia2},
$0=\1_{X_{sing}}\mu=p_*\1_{p^{-1}X_{sing}}\mu'$, and thus  $\mu=p_*\tilde\mu$
where 
$$
\tilde\mu:=\1_{p^{-1}(X\setminus X_{sing})}\mu'
$$
has support on the closure of $p^{-1}(X\setminus X_{sing})$
which is (contained in) $\widetilde X$. Thus the claim is proved.

Next we claim that $\tilde\mu=\tilde i_* \tilde\tau$ for a current $\tilde\tau$ on 
$\widetilde X$.  In fact, let
$\xi$ is a test form on $\widetilde Y$ such that $\tilde i^*\xi=0$. Since $p$ is
a biholomorphism outside $p^{-1}X_{sing}$, $\xi\w\tilde\mu=0$ there since $\mu=i_*\tau$ there.
Since $\tilde\mu$ has support on $\widetilde X$ it follows that
$\xi\w\tilde\mu=0$ outside $\widetilde X\cap p^{-1}X_{sing}$, and hence $\xi\w\tilde\mu=0$ by continuity.
Thus the claim follows.  

From the smooth case we know  that $\tilde\tau$ is \pmm and therefore  
 $p'_*\tilde\tau$ is \pmm as well.
%% then
%%$\xi\w\tilde\mu=0$ on $\widetilde Y\setminus (p')^{-1}X_{sing}$, and hence
%%it vanishes identically since $\1_{(p')^{-1}X_{sing}}\tilde\mu=0$.
%%Thus, by definition,  $\tilde\mu=\tilde i_* \tilde\tau$ for some current $\tilde\tau$, and from the smooth
%%case we know that $\tilde\tau $ is in $\PM(\widetilde X)$. 
%%Now,  $p'_*\tilde\tau$
%%is in $\PM(X)$,  and 
Finally,  $i_*p'_*\tilde\tau=p_*\tilde i_*\tilde\tau=p_*\tilde\mu=\mu=i_*\tau$
and thus $p'_*\tilde\tau=\tau$. Thus $\tau$ is pseudomeromorphic.   The second
part of (ii) is verified as the second part of (i).
\end{proof}


\begin{thebibliography}{99}





\bibitem{A1}\textsc{M.\ Andersson}
\textit{Residue currents and ideals of holomorphic functions}
\textit{Bull.\  Sci. Math.},  {\bf 128}, (2004), 481--512




%%\bibitem{Astrong} \textsc{M.\ Andersson:} A residue criterion for strong holomorphicity.
%%\textit{Ark. Mat.}, \textbf{48(1)} (2010), 1--15.

%%\bibitem{Aext} \textsc{M.\ Andersson:} Coleff-Herrera currents, duality, and Noetherian operators.
%%\textit{Bull. Soc. Math. France},  Bull. Soc. Math. France 139 (2011),  535--554.


%%\bibitem{ASarXiv} \textsc{M.\ Andersson, H.\ Samuelsson:} Koppelman formulas and the $\debar$-equation on an 
%%analytic space.
%%\textit{arXiv:0801.0710}%%Institut Mittag-Leffler Preprint Series}, REPORT No. 40, 2007/2008, spring.



\bibitem{AS} \textsc{M.\ Andersson, H.\ Samuelsson:}
 A Dolbeault-Grothendieck lemma on complex spaces via Koppelman formulas.
\textit{ Invent. Math.} \textbf{190} (2012), 261--297. 


%%\bibitem{AS} \textsc{M.\ Andersson, H.\ Samuelsson:} Koppelman formulas and the $\debar$-equation on an 
%%analytic space.
%%\textit{Institut Mittag-Leffler Preprint Series}, REPORT No. 40, 2007/2008, spring.

\bibitem{ASS} \textsc{M.\ Andersson, H.\ Samuelsson, J.\ Sznajdman:} On the Brian\c{c}on-Skoda theorem 
on a singular variety.
\textit{Ann. Inst. Fourier (Grenoble)}, \textbf{60(2)} (2010), 417--432.

\bibitem{AW1} \textsc{M.\ Andersson, E.\ Wulcan:} Residue currents with prescribed annihilator ideals.
\textit{Ann. Sci. \'{E}cole Norm. Sup.}, \textbf{40} (2007), 985--1007.

\bibitem{AW2} \textsc{M.\ Andersson, E.\ Wulcan:} Decomposition of residue currents.
\textit{J. Reine Angew. Math.}, \textbf{638} (2010), 103--118.

\bibitem{AWsemester} \textsc{M.\ Andersson, E.\ Wulcan:} On the effective membership problem on singular varieties. 
\textit{arXiv:1107.0388}.

%%\bibitem{Barlet} \textsc{D.\ Barlet:} Le faisceau $\omega_X$ sur un espace analytique $X$ de dimension
%%pure.
%%\textit{Fonctions de plusieurs variables complexes, III (S\'em. Fran\c cois Norguet, 1975--1977)}
%%187--204, Lecture Notes in Math., 670, Springer, Berlin, 1978. 


\bibitem{Bj}\textsc{J-E.\ Bj\"ork:}\textit{Residue currents
and $\D$-modules on complex manifolds}
{Preprint  Stockholm (1996)}



\bibitem{BS} \textsc{J-E.\ Bj\"ork, H.\ Samuelsson:}\textit{ Regularizations of residue currents.}
\textit{J. Reine Angew. Math.}, \textbf{640} (2010), 101--115.




\bibitem{BESVO}\textsc{Bravo, Ana Mar\'ia; Encinas, Santiago; Villamayor U., Orlando:}
\textit{ A simplified proof of desingularization and applications.}\textit{ Rev. Mat. Iberoamericana}
\textbf{ 21 } (2005), 349--458. 

\bibitem{CoHe}\textsc{Coleff, N.; Herrera, M.}
\textit{ Les courants r\'esiduels associ\'es \`a une forme m\'eromorphe. (French) [The residue currents associated with a meromorphic form]}
 \textit{Lecture Notes in Mathematics}
{\bf  633}. Springer, Berlin, 1978. x+211 pp. 



\bibitem{JP} \textsc{J.-P.\ Demailly:} Complex analytic and algebraic geometry.
\textit{Monograph}, Grenoble. Available at \url{http://www-fourier.ujf-grenoble.fr/~demailly/books.html}.



\bibitem{DFV} \textsc{K.\ Diederich, J.\ E.\ Forn\ae ss, S.\ Vassiliadou:} Local $L^2$
results for $\debar$ on a singular surface.
\textit{Math. Scand.}, \textbf{92} (2003), 269--294.



\bibitem{Eis} \textsc{D.\ Eisenbud:} Commutative algebra. With a view toward algebraic geometry.
\textit{Graduate Texts in Mathematics}, \textbf{150}. Springer-Verlag, New-York, 1995.



%%\bibitem{FoGa} \textsc{J.\ E.\ Forn\ae ss, E.\ A.\ Gavosto:} The Cauchy-Riemann equation on 
%%singular spaces.
%%\textit{Duke Math. J.}, \textbf{93} (1998), 453--477.




%%\bibitem{FOV} \textsc{J.\ E.\ Forn\ae ss, N.\ \O vrelid, S.\ Vassiliadou:} Semiglobal
%%results for $\debar$ on a complex space with arbitrary singularities.
%%\textit{Proc. Am. Math. Soc.}, \textbf{133(8)} (2005), 2377--2386.

 


\bibitem{HePa} \textsc{G.\ Henkin, M.\ Passare:} Abelian differentials on singular varieties and variations on a theorem
of Lie-Griffiths.
\textit{Invent. Math.}, \textbf{135(2)} (1999), 297--328.




\bibitem{HePo} \textsc{G.\ Henkin, P.\ Polyakov:} The Grothendieck-Dolbeault lemma for complete intersections.
\textit{C. R. Acad. Sci. Paris S\'{e}r I Math}, \textbf{308(13)} (1989), 405--409.




\bibitem{HePo2} \textsc{G.\ Henkin, P.\ Polyakov:} 
Residual d-bar-cohomology and the complex Radon transform on subvarieties of CPn.
\textit{Math.\ Ann.}, \textbf{354}(2012), 497--527.




% \bibitem{Hu} \textsc{C.\ Huneke:} Uniform bounds in Noetherian rings.
% \textit{Invent. Math.}, \textbf{107} (1992), 203--223.


\bibitem{Lark1} \textsc{R.\ L\"ark\"ang:} \textit{Residue currents associated with weakly holomorphic functions.}
\textit{ Ark. Mat.}  {\bf 50} (2012),  135--164. 


\bibitem{Lark2} \textsc{R.\ L\"ark\"ang:}
 \textit{On the duality theorem on an analytic variety.}
\textit{ Math. Ann.} {\bf 355} (2013), no. 1, 215--234. 32C30 (32A27)






%%\bibitem{LS} \textsc{R.\ L\"ark\"ang, H.\ Samuelsson:} Various approaches to 
%%products of residue currents. \textit{J.\ Functional Analysis} ?????????


\bibitem{Lund}  \textsc{J.\ Lundqvist}
\textit{ An effective uniform Artin-Rees lemma.}
\textit{arXiv:1306.5956 }




\bibitem{PTY}\textsc{M.\ Passare \& A.\ Tsikh \&  A.\ Yger}
 \textit{Residue currents of the Bochner-Martinelli type}
 \textit{Publ.\ Mat.}  {\bf 44} (2000), 85--117.



\bibitem{Sz}\textsc{J.\ Sznajdman}
\textit{A residue calculus approach to the uniform Artin-Rees lemma.}
\textit{ Israel J. Math.}
{\bf 196} (2013), no. 1, 33--50. 




%%\bibitem{LJ} \textsc{M.\ Lejeune-Jalabert:}
%%Remarque sur la classe fondamentale d'un cycle. 
%%\textit{C. R. Acad. Sci. Paris S\'er.  I Math.} \textbf{292} (1981), 801--804.







 


%%\bibitem{hasamArkiv} \textsc{H.\ Samuelsson:} Analytic continuation of residue currents.
%%\textit{Ark. Mat.}, \textbf{47(1)} (2009), 127--141.



 


 


\end{thebibliography}
\end{document}